\documentclass[12pt,reqno]{amsart}
\usepackage{latexsym,amsmath,amsfonts,amssymb,amsthm}
\def\Z{\mathbb Z}

\def\F{\mathbb F}

\def\Aut{{\rm{Aut\;}}}
\def\Inn{{\rm{Inn\;}}}

\theoremstyle{plain}
\newtheorem{Thm}{Theorem}
\newtheorem{Lem}{Lemma}
\newtheorem{Cor}{Corollary}

\theoremstyle{definition}

\theoremstyle{remark}

\begin{document}
\title{Restricted Sumsets in Finite Nilpotent Groups}
\author{Shanshan Du}
\email{ssdu.stand@gmail.com}
\author{Hao Pan}
\email{haopan79@yahoo.com.cn}
\address{Department of Mathematics, Nanjing University,
Nanjing 210093, People's Republic of China}
\begin{abstract}
Suppose that $A,B$ are two non-empty subsets of the finite nilpotent group $G$. If $A\not=B$, then the cardinality of the restricted sumset
$$A\dotplus B=\{a+b:\,a\in A,\ b\in B,\ a\neq b\}
$$
is at least
$$\min\{p(G),|A|+|B|-2\},$$
where $p(G)$ denotes the least prime factor of $|G|$.
\end{abstract}

\maketitle

\section{Introduction}
\setcounter{equation}{0}\setcounter{Thm}{0}\setcounter{Lem}{0}\setcounter{Cor}{0}

Suppose that $p$ is a prime and $A,B$ are two non-empty subsets of $\Z_p=\Z/p\Z$. The classical Cauchy-Davenport theorem (cf. \cite[Theorem 2.2]{N}) says that the sumset
$$
A+B=\{a+b:\,a\in A,\ b\in B\}
$$
contains at least
$$
\min\{p,|A|+|B|-1\}
$$
elements. In \cite{EH}, Erd\H os and Heilbronn considered the cardinality of the restricted sumset
$$
A\dotplus B=\{a+b:\,a\in A,\ b\in B,\ a\neq b\}.
$$
They conjectured that for non-empty $A\subseteq\Z_p$,
$$
|A\dotplus A|\geq\min\{p,2|A|-3\}.
$$
This conjecture was confirmed by Dias da Silva and Hamidoune \cite{DH}, with help of the exterior algebra. In 1996, using the polynomial method, Alon, Nathanson and Ruzsa \cite{ANR} gave a simple proof of the Erd\H os-Heilbronn conjecture.
In fact, they obtained a stronger result:
\begin{equation}\label{anr}
|A\dotplus B|\geq\min\{p,|A|+|B|-2\},
\end{equation}
provided $|A|\neq |B|$. Obviously, arbitrarily choosing $B\subseteq A$ with $|B|=|A|-1$, we have
$$
|A\dotplus A|\geq |A\dotplus B|\geq |A|+|B|-2=2|A|-3.
$$
Recently, K\'arolyi \cite{K4} considered the exceptional case that $|A|=|B|$  and $A\neq B$. He proved that (\ref{anr}) always holds as long as $A\neq B$.

On the other hand,  in \cite{O}, Olson proved that for any finite group $G$ and non-empty $A,B\subseteq G$, there exist a non-empty subset $C\subseteq A+B$ and a subgroup $H$ of $G$ with $H+C=C$ or $C+H=C$ such that
$$
|C|\geq |A|+|B|-|H|.
$$
Here for convenience, we still use ``$+$'', rather than ``$\times$'', to represent the binary operation over $G$.
It easily follows from Olson's result that
\begin{equation}\label{cdg}
|A+B|\geq\min\{p(G),|A|+|B|-1\},
\end{equation}
where $p(G)$ denotes the least prime factor of $|G|$. This is an extension of the Cauchy-Davenport theorem for finite groups.
In \cite{K1,K2}, K\'arolyi established the following generalizaton of the the Erd\H os-Heilbronn problem:
$$
|A\dotplus A|\geq\min\{p(G),2|A|-3\},
$$
where $A$ is a non-empty subset of the finite abelian group $G$.
Subsequently, Balister and Wheeler \cite{BW} removed the restriction that $G$ is abelian. In fact, they showed that
$$
|A\dotplus B|\geq\min\{p(G),|A|+|B|-3\},
$$
for any non-empty subsets $A,B$ of a finite group $G$.

In this paper, we shall consider the extension of (\ref{anr}) for finite nilpotent groups.
\begin{Thm}\label{t1}
Suppose that $G$ is a finite nilpotent group. Let $A,B$ be two non-empty subset of $G$. If $A\not=B$, then
$$|A\dotplus B|\geq\min\{p(G),|A|+|B|-2\}.$$
\end{Thm}
As we shall see later, in order to complete the proof of Theorem \ref{t1}, we need to discuss the structure of $A$ when $|A\dotplus A|=2|A|-3$.

For $A\subseteq\Z_p$ with $|A|<(p+1)/2$, if $|A+A|=2|A|-1$, then a theorem of Vosper (cf. \cite[Theorem 2.1]{N}) says that $A$ must be an arithmatic progression.
In \cite{K3}, K\'arolyi proved that if $A$ is a subset of the finite abelian group $G$ satisfying $5\leq |A|<(p(G)+3)/2$, then
$|A\dotplus A|=2|A|-3$ if and only if $A$ is an arithmatic progression. Now we shall prove that
\begin{Thm}\label{t2}
Let $G$ be a finite group and $A$ be a non-empty subset of $G$ with $|A|<(p(G)+3)/2$.
Suppose that
$$
|A\dotplus A|=2|A|-3.
$$
Then $A$ is commutative, i.e., $a_1+a_2=a_2+a_1$ for any $a_1,a_2\in A$.
\end{Thm}
In view of Theorem \ref{t2}, we know that if $|A\dotplus A|=2|A|-3$, then the subgroup generated by $A$ is actually abelian. Thus
\begin{Cor}\label{c1}
Under the assumptions of Theorem \ref{t2}, if $|A|=n\geq 5$, then $A=\{a,a+d,a+2d,\ldots,a+(n-1)d\}$ where $a,d\in G$ and $a+d=d+a$.
\end{Cor}

\section{Proof of Theorem \ref{t1}}
\setcounter{equation}{0}\setcounter{Thm}{0}\setcounter{Lem}{0}\setcounter{Cor}{0}

In this section, we shall give the most part of the proof of Theorem \ref{t1}, except for one subcase which requires Theorem \ref{t2}.
\begin{Lem}\label{hall} Suppose that $G$ is a finite group. Let $A=\{a_1,\ldots,a_n\}$ and $B=\{b_1,\ldots,b_m\}$ be two non-empty subsets of $G$ with $n+m-1\leq p(G)$. Then there exist $1\leq i_2,\ldots,i_n\leq m$ such that
$$
a_1+b_1,\ldots,a_1+b_m,a_2+b_{i_2},\ldots,a_{n}+b_{i_n}
$$
are all distinct.
\end{Lem}
\begin{proof} For $2\leq j\leq n$, let
$$
X_j=(\{a_1,a_j\}+B)\setminus(a_1+B).
$$
In view of (\ref{cdg}), for any non-empty $J\subseteq\{2,\ldots,n\}$,
$$
\bigg|\bigcup_{j\in J}X_j\bigg|=|((\{a_1\}\cup\{a_j\}_{j\in J})+B)\setminus(a_1+B)|\geq|J|.
$$
Applying the Hall marriage theorem (cf. \cite[Theorem 5.1]{VW}), we may choose distinct $c_2\in X_2,\ c_3\in X_3,\ \ldots, c_n\in X_n$. Letting $b_{i_j}=c_j-a_j$, we are done.
\end{proof}

Now suppose that $G$ is not a group of prime order.
Let $p=p(G)$. Without loss of generality, assume that $|A|+|B|-2\leq p$. In fact, if $|A|+|B|-2>p$, then we may choose non-empty $A'\subseteq A$ and $B'\subseteq B$ such that $|A'|+|B'|-2=p$. 
Clearly $A'\dotplus B'\subseteq A\dotplus B$. If $p=2$, then it is easy to check directly that $|A\dotplus B|\geq|A|+|B|-2$ provided $A\neq B$ and $|A|+|B|\leq 4$.
So we only need to consider those odd $p$.

If $G$ is abelian, let $H$ be a subgroup of $G$ such that $[G:H]=p$.
Otherwise, let $H$ be the center of $G$. Since $G$ is nilpotent, $G/H$ is also a non-trivial nilpotent group.
Below we assume that Theorem \ref{t1} holds for $H$ and $G/H$. 

For conveniece, let $\bar{a}$ denote the coset $a+H$.
Suppose that
$$
A=\bigcup_{j=1}^n(a_j+S_j),\qquad B=\bigcup_{j=1}^m(b_j+T_j),
$$
where $S_j,T_j$ are non-empty subsets of $H$ and $\bar{a}_i\neq \bar{a}_j$, $\bar{b}_i\neq \bar{b}_j$ for any $i\neq j$.
Furthermore, we may assumse that $\bar{a}_i=\bar{b}_j$ implies $a_i=b_j$.
Since either $G$ is abelian or $H$ is the center of $G$, we have $S+b=b+S$ for any $b\in G$ and $S\subseteq H$. Therefore
$$
A\dotplus B=\bigg(\bigcup_{\substack{1\leq i\leq n\\ 1\leq j\leq m\\ a_i=b_j}}(a_i+b_j+(S_i\dotplus T_j))\bigg)\cup\bigg(\bigcup_{\substack{1\leq i\leq n\\ 1\leq j\leq m\\ a_i\neq b_j}}(a_i+b_j+(S_i+T_j))\bigg).
$$
Let $\bar{A}=\{\bar{a}_1,\ldots,\bar{a}_n\}$ and $\bar{B}=\{\bar{b}_1,\ldots,\bar{b}_m\}$. It trivially follows that
\begin{equation}\label{AdB}
|A\dotplus B|\geq |\bar{A}\dotplus \bar{B}|-1+\max_{\substack{1\leq i\leq n\\ 1\leq j\leq m \\ a_i\neq b_j}}|S_i+T_j|.
\end{equation}

For $S\subseteq H$ and $a\in G$, clearly
$$
-(a+S)=(-S)+(-a)=(-a)+(-S).$$
So we can ``exchange'' $A$ and $B$ in the sense
$$
|A\dotplus B|=|-(A\dotplus B)|=|(-B)\dotplus(-A)|.
$$
That is, we may assume $m\geq n$. Furthermore, when $m=n$, assume that
$$
\max\{|S_1|,\ldots,|S_n|\}\geq\max\{|T_1|,\ldots,|T_m|\}.
$$

If $n+m-1>p$, then $|\bar{A}|+|\bar{B}|=|A|+|B|=p+2$ by recalling $|A|+|B|-2\leq p$.
Since $p$ is odd, we must have $\bar{A}\not=\bar{B}$.
In view of (\ref{AdB}), we get
$$
|A\dotplus B|\geq|\bar{A}\dotplus \bar{B}|\geq |\bar{A}|+|\bar{B}|-2.
$$

Below we always assume that $n+m-1\leq p$.
Suppose that
$|S_1|\geq|S_2|\geq\cdots |S_n|$.

\medskip\smallskip\noindent(i) $m>n$ and $|S_1|\geq 2$, or $m=n$ and $|S_1|\geq 3$.\medskip\smallskip

Applying Lemma \ref{hall} for $\bar{A}+\bar{B}$,  we know there exist
$1\leq i_2,\ldots,i_n\leq m$ such that
$$
\bar{a}_1+\bar{b}_1,\ldots,\bar{a}_1+\bar{b}_m,\bar{a}_2+\bar{b}_{i_2},\ldots,\bar{a}_{n}+\bar{b}_{i_n}.
$$
are distinct elements of $\bar{A}+\bar{B}$. Without loss of generality, assume that $a_1\not\in\{b_2,\ldots,b_m\}$. Then
\begin{align}\label{ABg}
|A\dotplus B|\geq
&\bigg|\bigcup_{j=1}^m((a_1+S_1)\dotplus(b_j+T_j))\cup
\bigcup_{j=2}^n((a_j+S_j)\dotplus(b_{i_j}+T_{i_j}))\bigg|\notag\\
\geq&|a_1+b_1+(S_1\dotplus T_1)|+\sum_{j=2}^m|a_1+b_j+(S_1+T_j)|+
\sum_{j=2}^n|a_j+b_{i_j}+(S_j\dotplus T_{i_j})|\notag\\
\geq&|S_1|+|T_1|-3+\sum_{j=2}^m(|S_1|+|T_j|-1)+
\sum_{j=2}^n|S_j\dotplus T_{i_j}|\notag\\
\geq&|B|-2+m(|S_1|-1)+
\sum_{j=2}^n|S_j\dotplus T_{i_j}|.
\end{align}
It is easy to see that
$|S\dotplus T|\geq |S|-1$. Hence for $2\leq j\leq n$,
$$
|S_1|-1+|S_j\dotplus T_{i_j}|\geq |S_1|+|S_j|-2\geq |S_j|.
$$
If $m>n$, then
$$
|A\dotplus B|\geq |B|-3+(m-n)+|S_1|+\sum_{j=2}^n(|S_1|-1+|S_j\dotplus T_{i_j}|)\geq |A|+|B|-2.
$$
If $m=n=1$, then we have $S_1\neq T_1$ and
$$
|A\dotplus B|=|S_1\dotplus T_1|\geq |S_1|+|T_1|-2=|A|+|B|-2.
$$
Suppose that $m=n\geq 2$. Then since $|S_1|\geq 3$,
$$
|S_1|-1+|S_j\dotplus T_{i_j}|\geq |S_j|+1.
$$
In view of (\ref{ABg}),
$$
|A\dotplus B|\geq |B|-3+|S_1|+\sum_{j=2}^n(|S_j|+1)\geq |A|+|B|-2.
$$

\medskip\smallskip\noindent(ii) $m>n$ and $|S_1|=1$.
\medskip\smallskip

Clearly now $\bar{A}\neq \bar{B}$. Hence $|\bar{A}\dotplus \bar{B}|\geq n+m-2$. We need to consider three cases.

\medskip\smallskip(1) Suppose that $\bar{A}\not\subseteq \bar{B}$. In particular, we may assume that $\bar{a}_1\not\in \bar{B}$. Then for some $2\leq i_2,\ldots,i_{n-1}\leq n$ and $1\leq k_2,\ldots,k_{n-1}\leq m$,
$$
\bar{a}_1+\bar{b}_1,\bar{a}_1+\bar{b}_2,\ldots,\bar{a}_1+\bar{b}_m,\bar{a}_{i_2}+\bar{b}_{k_2},\ldots,\bar{a}_{i_{n-1}}+\bar{b}_{k_{n-1}}
$$
are distinct elements of $\bar{A}\dotplus \bar{B}$, where $\bar{a}_{i_j}\neq \bar{b}_{k_j}$ for $2\leq j\leq n-1$. Thus
\begin{align*}
|A\dotplus B|\geq\sum_{j=1}^{m}|S_1+T_j|
+\sum_{j=2}^{n-1}|S_{i_j}+ T_{k_j}|\geq\sum_{j=1}^{m}|T_j|
+n-2=|A|+|B|-2.
\end{align*}

\medskip\smallskip(2) Suppose that $\bar{A}\subseteq \bar{B}$ and $\bar{a}_j+\bar{a}_j\in \bar{A}\dotplus \bar{B}$ for all $1\leq j\leq n$. Without loss of generality, assume that $a_1=b_1$.
Clearly now $\bar{A}\dotplus \bar{B}=\bar{A}+\bar{B}$, i.e., $|\bar{A}\dotplus \bar{B}|\geq n+m-1$. Hence we have
$$\bar{a}_1+\bar{b}_1,\bar{a}_1+\bar{b}_2,\ldots,\bar{a}_1+\bar{b}_m,\bar{a}_{i_2}+\bar{b}_{k_2},\ldots,\bar{a}_{i_{n}}+\bar{b}_{k_{n}}
$$
are distinct elements of $\bar{A}\dotplus \bar{B}$, where $\bar{a}_{i_j}\neq \bar{b}_{k_j}$. It follows that
\begin{align*}
|A\dotplus B|\geq&|S_1\dotplus T_1|+\sum_{j=2}^{m}|S_1+T_j|
+\sum_{j=2}^{n}|S_{i_j}+T_{k_j}|\\
\geq&(|T_1|-1)+\sum_{j=2}^{m}|T_j|
+n-1=|A|+|B|-2.
\end{align*}

\medskip\smallskip(3) Suppose that $\bar{A}\subseteq \bar{B}$ but $\bar{a}_{i_0}+\bar{a}_{i_0}\not\in \bar{A}\dotplus \bar{B}$ for some $1\leq i_0\leq n$. 
Without loss of generality, assume that $\bar{a}_1+\bar{a}_1\not\in\bar{A}+\bar{B}$ and $a_1=b_1$.
Since $|\bar{A}\dotplus \bar{B}|\geq n+m-2$, we may assume that
$$
\bar{a}_{1}+\bar{b}_2,\ldots,\bar{a}_{1}+\bar{b}_m,\ \bar{a}_{i_2}+\bar{b}_{k_2},\ldots,\bar{a}_{i_{n}}+\bar{b}_{k_{n}}
$$
are distinct elements of $\bar{A}\dotplus \bar{B}$. We still have
\begin{align*}
|A\dotplus B|\geq&|S_1\dotplus T_1|+\sum_{j=2}^{m}|S_1+T_j|
+\sum_{j=2}^{n}|S_{i_j}+T_{k_j}|\geq|A|+|B|-2.
\end{align*}

\medskip\smallskip\noindent(iii) $m=n$ and $|S_1|=2$.
\medskip\smallskip

The case $m=n=1$ is trivial. Suppose that $m=n\geq 2$.
If $\bar{a}_1\not\in \bar{B}$, then following the same discussion in the first case of (ii), we can get
$|A\dotplus B|\geq |A|+|B|-2$.

Suppose that $\bar{a}_1\in \bar{B}$ and $a_1=b_1$. In view of Lemma \ref{hall}, we may assume that
$$
\bar{a}_1+\bar{b}_1,\bar{a}_1+\bar{b}_2,\ldots,\bar{a}_1+\bar{b}_n,\bar{a}_{2}+\bar{b}_{i_2},\ldots,\bar{a}_{n}+\bar{b}_{i_n}
$$
are distinct elements of $\bar{A}+\bar{B}$. Suppose that $S_1\neq T_1$. Then
\begin{align*}
|A\dotplus B|\geq&|S_1\dotplus T_1|+\sum_{j=2}^{n}|S_1+T_j|
+\sum_{j=2}^{n}|S_{j}\dotplus T_{i_j}|\\
\geq&|S_1|+|T_1|-2+\sum_{j=2}^{n}(|S_1|+|T_j|-1)
+\sum_{j=2}^{n}(|S_{j}|-1)\\
=&|A|+|B|+(n-1)|S_1|-n-n=|A|+|B|-2.
\end{align*}
Now suppose that $S_1=T_1$. Since $A\neq B$, there exists some $2\leq j_0\leq n$ such that either $a_{j_0}\neq b_{i_{j_0}}$ or $S_{j_0}\neq T_{i_{j_0}}$

\medskip\smallskip(1) Suppose that $a_{j_0}\neq b_{i_{j_0}}$ for some $2\leq j_0\leq n$. Then
\begin{align*}
&|A\dotplus B|\geq|S_1\dotplus T_1|+\sum_{j=2}^{n}|S_1+T_j|
+|S_{j_0}+T_{i_{j_0}}|+\sum_{\substack{2\leq j\leq n\\ j\neq j_0}}|S_{j}\dotplus T_{i_j}|\\
\geq&n|S_1|+|B|-(n+2)+|S_{j_0}|+|T_{i_{j_0}}|-1+\sum_{\substack{2\leq j\leq n\\ j\neq j_0}}(|S_{j}|-1)\geq |A|+|B|-2.
\end{align*}

\medskip\smallskip(2) Suppose that $a_j=b_{i_j}$ for all $2\leq j\leq n$ and $S_{j_0}\neq T_{i_{j_0}}$ for some $2\leq j_0\leq n$.
If $|S_{j_0}|=2$, then we may exchange $a_{1}$ and $a_{j_0}$. Thus the desired result follows from our discussion on the case $S_1\neq T_1$.
Assume that $|S_{j_0}|=1$. Since $S_{j_0}\neq T_{i_{j_0}}$, $S_{j_0}\dotplus T_{i_{j_0}}$ is non-empty.
Then
\begin{align*}
&|A\dotplus B|\geq|S_1\dotplus T_1|+\sum_{j=2}^{n}|S_1+T_j|
+|S_{j_0}\dotplus T_{i_{j_0}}|+\sum_{\substack{2\leq j\leq n\\ j\neq j_0}}|S_{j}\dotplus T_{i_j}|\\
\geq&n|S_1|+|B|-(n+2)+|S_{j_0}|+\sum_{\substack{2\leq j\leq n\\ j\neq j_0}}(|S_{j}|-1)\geq |A|+|B|-2.
\end{align*}

\medskip\smallskip\noindent(iv) $m=n$ and $|S_1|=1$.
\medskip\smallskip

Recalling (\ref{AdB}), we have
$$
|A\dotplus B|\geq |\bar{A}\dotplus \bar{B}|\geq|A|+|B|-2,
$$
if $\bar{A}\neq \bar{B}$ or $\bar{A}\dotplus \bar{B}=\bar{A}+\bar{B}$.
So we may assume that $a_j=b_j$  for $1\leq j\leq n$ and $\bar{a}_1+\bar{a}_1\not\in \bar{A}\dotplus \bar{B}$. If $S_1\not=T_1$, then
$$
|A\dotplus B|\geq |S_1\dotplus T_1|+|\bar{A}\dotplus \bar{B}|\geq 1+(|\bar{A}|+|\bar{B}|-3)=|A|+|B|-2.
$$

However, the final case $S_1=T_1$ is most annoying. In fact, its proof needs Theorem \ref{t2} and Corollary \ref{c1}.
So we shall firstly prove Theorem \ref{t2} in Section 3, before completing the proof of Theorem \ref{t1}.

\section{Proof of Theorem \ref{t2} for finite nilpotent groups}
\setcounter{equation}{0}\setcounter{Thm}{0}\setcounter{Lem}{0}\setcounter{Cor}{0}

In this section, we shall only prove Theorem \ref{t2} for finite nilpotent groups, which is sufficient to complete the proof of Theorem \ref{t1}.

Suppose that $G$ is a finite non-abelian nilpotent group and $H$ is the center of $G$. Assume that Theorem \ref{t2} holds for $G/H$.
Let $A$ be a non-empty subset of $G$ satisfying
$$
|A\dotplus A|=2|A|-3.
$$
We shall prove that $A$ is commutative.
Assume that
$$
A=\bigcup_{i=1}^n(a_i+S_i),
$$
where $\emptyset\neq S_i\subseteq H$ and $\bar{a}_i\not=\bar{a}_j$ for $i\neq j$.
There is nothing to do when $n=1$. Below assume that $n\geq 2$. Furthermore, if $p(G)=2$ and $A=\{a_1,a_2\}$, then $|A\dotplus A|=1$ if and only if $a_1+a_2=a_2+a1$.
So we may assume that $p(G)$ is odd.

\medskip\smallskip Suppose that $|S_1|\geq|S_2|\geq\cdots\geq|S_m|$. By (\ref{ABg}), it is impossible that $|S_1|\geq 3$.
Assume that $|S_1|=2$ and
$$
\bar{a}_1+\bar{a}_1,\bar{a}_1+\bar{a}_2,\ldots,\bar{a}_1+\bar{a}_n,\bar{a}_{2}+\bar{a}_{i_2},\ldots,\bar{a}_{n}+\bar{a}_{i_n}
$$
are distinct elements of $\bar{A}+\bar{A}$. If $a_{j_0}\neq a_{i_{j_0}}$ for some $2\leq j_0\leq n$, then
\begin{align*}
&|A\dotplus A|\geq|S_1\dotplus S_1|+\sum_{j=2}^{n}|S_1+S_j|
+|S_{j_0}+S_{i_{j_0}}|+\sum_{\substack{2\leq j\leq n\\ j\neq j_0}}|S_{j}\dotplus S_{i_j}|\geq 2|A|-2.
\end{align*}
So we must have $j=i_j$ for all $2\leq j\leq n$.
Now
\begin{equation}\label{subset}
(a_1+a_1+(S_1\dotplus S_1))\cup\bigcup_{j=2}^n(a_1+a_j+(S_1+S_j))\cup\bigcup_{j=2}^n(a_j+a_j+(S_j\dotplus S_j))
\end{equation}
contains at least
\begin{equation*}
|S_1\dotplus S_1|+\sum_{j=1}^n|S_1+S_j|+\sum_{j=2}^n|S_j\dotplus S_j|\geq
\sum_{j=1}^n(|S_1|+|S_j|-1)-2+\sum_{j=2}^n(|S_j|-1)\geq 2|A|-3
\end{equation*}
elements. That is, the set (\ref{subset}) shloud concide with $A\dotplus A$.
If there exists an element of $A\dotplus A$ not lying in (\ref{subset}), then we get a contradiction.\medskip\smallskip

Assume that there exist distinct $1\leq j_1,j_2\leq n$ satisfying
$$\bar{a}_{j_1}+\bar{a}_{j_2}\not\in\{\bar{a}_1+\bar{a}_1,\bar{a}_1+\bar{a}_2,\cdots,\bar{a}_1+\bar{a}_n,\bar{a}_2+\bar{a}_2,\cdots,\bar{a}_n+\bar{a}_n\}.$$
Then $a_{j_1}+a_{j_2}+(S_{j_1}+S_{j_2})$ is a non-empty subset of $A\dotplus A$. But it is evidently not included in (\ref{subset}). Therefore we may assume that
$$
\bar{A}\dotplus \bar{A}\subseteq\{\bar{a}_1+\bar{a}_1,\bar{a}_1+\bar{a}_2,\cdots,\bar{a}_1+\bar{a}_n,\bar{a}_2+\bar{a}_2,\cdots,\bar{a}_n+\bar{a}_n\}.
$$
Let
$$
J_1=\{1\leq j\leq n:\, \bar{a}_j+\bar{a}_j\in (\bar{A}\dotplus \bar{A})\setminus\{\bar{a}_1+\bar{a}_2,\cdots,\bar{a}_1+\bar{a}_n\}\},
$$
and let
$J_2=\{2,3,\ldots,n\}\setminus J_1$.

\medskip\smallskip(a) Suppose that there exists some $j_0\in J_1$ satisfying $|S_{j_0}|=1$. Since $\bar{a}_{j_0}+\bar{a}_{j_0}\in \bar{A}\dotplus \bar{A}$, we can find distinct $2\leq j_1,j_2\leq m$ such that $\bar{a}_{j_1}+\bar{a}_{j_2}=\bar{a}_{j_0}+\bar{a}_{j_0}$.
But now $S_{j_0}\dotplus S_{j_0}=\emptyset$ and $S_{j_1}+S_{j_2}\neq\emptyset$. Hence $a_{j_1}+a_{j_2}+(S_{j_1}+S_{j_2})\subseteq A\dotplus A$ is not included in (\ref{subset}).

\medskip\smallskip(b) Assume that $|S_j|=2$ for each $j\in J_1$.
Note that 
$
|J_1|\geq|\bar{A}\dotplus\bar{A}|-(n-1)$.
If $|\bar{A}\dotplus\bar{A}|>2n-3$, then $|J_1|\geq n-1$ and $|J_2|\leq 1$. And if 
$|\bar{A}\dotplus\bar{A}|=2n-3$, then by the induction hypothesis, $|J_1|=n-2$, $|J_2|\leq 2$ and $\bar{A}$ is commutative.
We may always find
$j_0\in J_1$ and $2\leq j_1,j_2\leq n$ such that $\bar{a}_{j_1}+\bar{a}_{j_2}=\bar{a}_{j_0}+\bar{a}_{j_0}$ and $\max\{|S_{j_1}|,|S_{j_2}||\}=2$ in the case $n\geq 4$. Since
$|S_{j_0}\dotplus S_{j_0}|=1$ and $|S_{j_1}+S_{j_2}|\geq 2$, we have $a_{j_1}+a_{j_2}+(S_{j_1}+S_{j_2})$ is not a subset of (\ref{subset}).

\medskip\smallskip Thus combining (a) and (b), we get that $|S_1|=2$ is impossible when $n\geq 4$.

\medskip\smallskip Now consider the case $n=3$. Suppose that
$|\bar{A}\dotplus \bar{A}|\geq 4$, i.e.,
$$
\bar{A}\dotplus \bar{A}\supseteq\{\bar{a}_1+\bar{a}_2,\bar{a}_1+\bar{a}_3,\bar{a}_{j_1}+\bar{a}_{k_1},\bar{a}_{j_2}+\bar{a}_{k_2}\}
$$
where $j_1\neq k_1$, $j_2\neq k_2$. Then
$$
|A\dotplus A|\geq(|S_1|+|S_2|-1)+(|S_1|+|S_3|-1)+2((|S_2|+|S_3|-1))\geq2|A|-2,
$$
since $|S_1|\geq|S_2|\geq|S_3|$. So we must have $|\bar{A}\dotplus \bar{A}|=3$. By the induction hypothesis, $\bar{A}$ is commutative, i.e., $\bar{a}_i+\bar{a}_j=\bar{a}_j+\bar{a}_i$
for $1\leq i\leq j\leq 3$. Hence
$$
\bar{A}\dotplus \bar{A}=\{\bar{a}_1+\bar{a}_2,\bar{a}_1+\bar{a}_3,\bar{a}_{2}+\bar{a}_{3}\}.
$$
Below we shall show that $\{a_1,a_2,a_3\}$ is actually commutative in $G$, which evidently implies $A$ is also commutative.
Assume that $a_2+a_1=a_1+a_2+h$ where $h\in H$. Note that
$$
(a_1+a_2+(S_1+S_2))\cup(a_1+a_3+(S_1+S_3))\cup(a_2+a_3+(S_2+S_3))
$$
contains exactly
\begin{equation}\label{set2}
(|S_1|+|S_2|-1)+(|S_1|+|S_3|-1)+(|S_2|+|S_3|-1)=2|A|-3
\end{equation}
elements. So we must have
$$a_1+a_2+(S_1+S_2)=a_2+a_1+(S_2+S_1)=a_1+a_2+h+(S_1+S_2),$$
i.e.,
$h+(S_1+S_2)=S_1+S_2$. Hence $S_1+S_2$ includes a coset of the subgroup generated by $h$.
However, since $|S_1+S_2|<p(G)$, this is impossible unless $h=0$.
Similarly, we can get $a_1+a_3=a_3+a_1$ and $a_2+a_3=a_3+a_2$.

The case $n=2$ is similar. In fact, $|\bar{A}+\bar{A}|=2|\bar{A}|-3$ implies that $\bar{A}$ is commutative. And from $a_2+a_1+(S_2+S_1)=a_1+a_2+(S_1+S_2)$,
we can deduce that $a_1+a_2=a_2+a_1$.

\medskip\smallskip Finally, suppose that $|S_1|=\cdots=|S_n|=1$. From $|A\dotplus A|=2|A|-3$, it follows that
$|\bar{A}\dotplus \bar{A}|=2|\bar{A}|-3$. By the induction hypothesis, $\bar{A}$ is commutative. If $a_i+a_j\neq a_j+a_i$ for some $i\neq j$, then by the above discussion, we know
$$
a_i+a_j+(S_i+S_j)\neq a_j+a_i+(S_i+S_j),
$$
i.e., the coset $a_i+a_j+H$ contains two elements of $A\dotplus A$.
Hence
$$
|A\dotplus A|\geq |\bar{A}\dotplus \bar{A}|+1=2|A|-2.
$$
This leads a contradiction. Thus $\{a_1,a_2,\ldots,a_n\}$ is commutative, as well as $A$.

The proof of Theorem \ref{t2} for finite nilpotent groups is concluded. \qed

\medskip\smallskip Let us return the proof of the final case of Theorem \ref{t1}. Suppose that
$$
A=\bigcup_{j=1}^n(a_j+S_j),\qquad B=\bigcup_{j=1}^n(a_j+T_j),
$$
where $|S_1|=\cdots=|S_n|=|T_1|=\cdots=|T_n|=1$ and $S_j=T_j$ if $\bar{a}_j+\bar{a}_j\not\in\bar{A}\dotplus\bar{B}$.
We need to show that $|A\dotplus B|\geq |A|+|B|-2$ if $S_{j_0}\neq T_{j_0}$ for some $1\leq j_0\leq n$.

Assume on the contrary that $|A\dotplus B|=|A|+|B|-3$. Then $|\bar{A}\dotplus\bar{B}|=|\bar{A}\dotplus\bar{A}|=2|\bar{A}|-3$. If $n\geq 5$, then by Corollary \ref{c1}, $\bar{A}$ is an arithmatic
progression. Suppose that $n=4$, i.e., $\bar{A}=\{\bar{a}_1,\bar{a}_2,\bar{a}_3,\bar{a}_4\}$. Clearly
$$
\bar{A}\dotplus\bar{A}\subseteq\{\bar{a}_1+\bar{a}_2,\bar{a}_1+\bar{a}_3,\bar{a}_1+\bar{a}_4,\bar{a}_2+\bar{a}_3,\bar{a}_2+\bar{a}_4,\bar{a}_3+\bar{a}_4\}.
$$
Noting that $|\bar{A}\dotplus\bar{A}|=5$, we may assume that $\bar{a}_1+\bar{a}_4=\bar{a}_2+\bar{a}_3$. Since $S_{j_0}\neq T_{j_0}$ for some $1\leq j_0\leq n$, we have
$\bar{a}_{j_0}+\bar{a}_{j_0}\in\bar{A}\dotplus\bar{A}$, i.e., there exist $1\leq j_1<j_2\leq 4$ such that $\bar{a}_{j_0}+\bar{a}_{j_0}=\bar{a}_{j_1}+\bar{a}_{j_2}$. Hence $\{\bar{a}_{j_1},\bar{a}_{j_0},\bar{a}_{j_2}\}$ forms an
arithmatic progression, as well as $\bar{A}$. Similarly, when $n=3$, we also can get $\bar{A}$ is an arithmatic progression.

Now we have proved that $\bar{A}=\{\bar{a},\bar{a}+\bar{d},\cdots,\bar{a}+(n-1)\bar{d}\}$. Since $\bar{A}$ is commutative, $\bar{a}+\bar{d}=\bar{d}+\bar{a}$. Suppose that
$d+a=a+d+h$ where $h\in H$. Without loss of generality, assume that $a_i=a+(i-1)d$. Clearly now
$$
\bar{A}+\bar{A}=\{\bar{a}+\bar{a}+\bar{d},\bar{a}+\bar{a}+2\bar{d},\cdots,\bar{a}+\bar{a}+(2n-3)\bar{d}\}.
$$
Since $2n-3<p(G)$, we have $\bar{a}+\bar{a}\not\in\bar{A}+\bar{A}$.
It follows from our assumption that $S_1=T_1$.
Suppose that $S_j=\{s_j\}$ and $T_j=\{t_j\}$ for $1\leq j\leq n$. Then
\begin{align*}
&a_j+a_1+(s_j+t_1)=a+(j-1)d+a+(s_j+t_1)\\
=&a+a+(j-1)d+(j-1)h+(s_j+t_1)\\
=&a_1+a_j+(s_1+t_j)=a+a+(j-1)d+(s_1+t_j).
\end{align*}
Hence $(j-1)h+(s_j+t_1)=s_1+t_j$ for $2\leq j\leq n$. Since $s_1=t_1$ and $s_{j_0}\neq t_{j_0}$ for some $1\leq j_0\leq n$, we must have $h\neq 0$. Thus
$s_j\neq t_j$ for each $2\leq j\leq n$. On the other hand, since $|\bar{A}\dotplus\bar{A}|=2n-3\leq |\bar{A}+\bar{A}|-2$, 
there exists $2\leq j_1\leq n$ such that $\bar{a}_{j_1}+\bar{a}_{j_1}\not\in\bar{A}\dotplus\bar{A}$. 
By our assumption, we should have $s_{j_1}=t_{j_1}$, which leads an evident contradiction.
All are done. \qed

\section{Proof of Theorem \ref{t2}: Generalized restricted sumsets}
\setcounter{equation}{0}\setcounter{Thm}{0}\setcounter{Lem}{0}\setcounter{Cor}{0}

In the next two sections, we shall complete the proof of Theorem \ref{t2} for general finite groups. Let $\Aut G$ denote the automorphism group of $G$.
For $\sigma\in\Aut G$ and $A,B\subseteq G$, define
$$
A\stackrel{\sigma}+B=\{a+b:\,a\in A,\ b\in B,\ a\neq\sigma(b)\}.
$$
In \cite{BW}, Balister and Wheeler proved
\begin{equation}\label{sigmaeh}
|A\stackrel{\sigma}+B|\geq\min\{p(G)-\delta,|A|+|B|-3\},
\end{equation}
where $\delta=1$ or $0$ according to whether the order of $\sigma$ is even or not.
For $A\subseteq G$, define
$$
\sigma(A)=\{\sigma(a):\,a\in A\}.
$$
Here we shall prove a generalizaton of Theorem \ref{t2}.
\begin{Thm}\label{t3}
Suppose that $G$ is a finite group and $A$ is a non-empty subsets of $G$. 
Let $\sigma$ be an automorphism of $G$ with odd order.
If $2|A|-3<p(G)$ and
$$
|\sigma(A)\stackrel{\sigma}+A|=2|A|-3,
$$
then $A$ is $\sigma$-commutative, i.e.,
$$
\sigma(a_1)+a_2=\sigma(a_2)+a_1.
$$
for any $a_1,a_2\in A$.
\end{Thm}
It is easy to verify Theorem \ref{t3} when $p(G)=2$. So below we always assume that $|G|$ is odd. From the well-known
Feit-Thompson theorem \cite{FT}, we know that $G$ is a solvable group.

For $a\in G$, define $\tau_a:G\to G$ by $\tau_a(x)=-a+x+a$ for any $x\in G$.  
Apparently $\tau_a\in\Aut G$. And $x\neq\sigma(y)$ if and only if $\tau_b(x)\neq\tau_b\sigma(y)$.
Let $\Inn G=\{\tau_a:\,a\in G\}$ be the inner automorphism group of $G$. 
We know that $\Inn G\cong G$ and $\Inn G\trianglelefteq \Aut G$. By the second isomorphism theorem,
$$
\langle\sigma\rangle \Inn G/\Inn G\cong \langle\sigma\rangle/(\langle\sigma\rangle\cap \Inn G),
$$
where $\langle\sigma\rangle$ is the subgroup generated by $\sigma$.
Hence if $\sigma$ is odd, then $\tau_a\sigma$ is also odd for any $a\in G$.

Suppose that $H$ is a normal subgroup of $G$ satisfying $\sigma(H)=H$. Then
for any coset $\bar{a}=a+H$, we have
$$
\sigma(\bar{a})=\sigma(a+H)=\sigma(a)+H.
$$
So $\sigma$ also can be viewed as an automorphism of $G/H$. The following lemma of Balister and Wheeler says that such $H$ always exists. 
For a prime power $p^\alpha$, let $\F_{p^\alpha}$ denote the finite field with $p^\alpha$ elements. 
\begin{Lem}[{\cite[Theorem 3.2]{BW}}]\label{bwl}
Suppose that $G$ is a finite solvable group and $\sigma$ is an automorphism of $G$. Then there exists
a proper normal subgroup $H$ of $G$ satisfying that

\noindent (i) $\sigma(H)=H$.

\noindent (ii) $G/H$ is isomorphic to the additive group of some finite field $\F_{p^\alpha}$.

\noindent (iii) Let $\chi$ denote the isomorphism from $G/H$ to the additive group of $\F_{p^\alpha}$.
Then there exists some $\gamma\in\F_{p^\alpha}\setminus\{0\}$ such that 
$\chi(\sigma(\bar{a}))=\gamma\cdot\chi(\bar{a})$ for each $\bar{a}\in G/H$.
\end{Lem}
The next lemma is a simple application of Alon's combinatorial nullstellensatz.
\begin{Lem}\label{field}
Let $A,B$ be two non-empty subsets of $\F_{p^\alpha}$ with $|A|=|B|$. Suppose that $\gamma\in\F_{p^\alpha}\setminus\{0,1\}$. Then
the cardinality of the restricted sumset
$$
A\stackrel{\gamma}+B=\{a+b:\, a\in A,\ b\in B,\ a\neq\gamma b\}
$$
is at least
$$
\min\{p,|A|+|B|-2\}.
$$
\end{Lem}
\begin{proof} Without loss of generality, assume that $|A|=|B|\geq 2$ and $|A|+|B|-2\leq p$. Assume on the contrary that $|A\stackrel{\gamma}+B|<|A|+|B|-2$.
Define the polynomial
$$
F(x,y)=(x-\gamma y)(x+y)^{|A|+|B|-3-|A\stackrel{\gamma}+B|}\prod_{c\in A\stackrel{\gamma}+B}(x+y-c).
$$
Clearly $\deg F(x,y)=|A|+|B|-2$ and $F(x,y)$ vanishes over the Casterian product $A\times B$. 
Let $[x^{n}y^{m}]F(x,y)$ denote the coefficient of $x^{n}y^{m}$ in the expansion of $F(x,y)$. By \cite[Theorem 1.2]{A},
$[x^{|A|-1}y^{|B|-1}]F(x,y)$ must be zero. On the other hand, clearly 
\begin{align*}
[x^{|A|-1}y^{|B|-1}]F(x,y)=&[x^{|A|-1}y^{|B|-1}]F(x,y)(x-\gamma y)(x+y)^{|A|+|B|-3}\\
=&\binom{|A|+|B|-3}{|A|-2}-\gamma\binom{|A|+|B|-3}{|B|-2}\\=&
(|A|+|B|-3)\binom{|A|+|B|-4}{|A|-2}\bigg(\frac1{|B|-1}-\frac\gamma{|A|-1}\bigg).
\end{align*}
Since $|A|=|B|$ and $\gamma\neq 1$, $[x^{|A|-1}y^{|B|-1}]F(x,y)$ doesn't vanish. This leads a contradiction.
\end{proof}
Let $H$ be a normal subgroup of $G$ satisfying the requirments of Lemma \ref{bwl}. 
Suppose that $|H|=1$. Then $G$ is isomorphic to the additive group of some $\F_{p^\alpha}$.
Let $\chi$ be the isomorphism from $G$ to $\F_{p^\alpha}$. In view of Lemma \ref{bwl}, there exists $0\neq\gamma\in\F_{p^\alpha}$
such that $\chi(\sigma(a))=\gamma\cdot\chi(a)$ for any $a\in G$. Hence applying Lemma \ref{field}, for $\emptyset\neq A\subseteq G$, we have
$$
|\sigma(A)\stackrel{\sigma}{+}A|=|\gamma\cdot\chi(A)\stackrel{\gamma}{+}\chi(A)|\geq\min\{p(G),2|A|-2\},
$$
unless $\sigma$ is the identity automorphism. Of course, if $\sigma$ is the identity automorphism, then clearly Theorem \ref{t3} is true since $G$ is abelian now.

Below assume that $|H|>1$ and Theorem \ref{t3} holds for $H$ and $G/H$. Note that for $a,b\in G$ and $S,T\subseteq H$,
$$
(a+S)+(b+T)=a+b+(-b)+S+b+T=a+b+(\tau_b(S)+T).
$$
And we have
$$
(\sigma(a)+S)\stackrel{\sigma}+(a+T)=\sigma(a)+a+(\tau_a(S)\stackrel{\tau_a\sigma}+T),
$$
Similarly as the proof of Theorem \ref{t2}, write
$$
A=\bigcup_{j=1}^n(a_j+S_j).
$$
where those $S_j$ are non-empty subsets of $H$. 
Now $\sigma(A)\stackrel{\sigma}+A$ equals to
$$
\bigg(\bigcup_{\substack{1\leq i\leq n}}(\sigma(a_i)+a_i+(\tau_{a_i}\sigma(S_i)\stackrel{\tau_{a_i}\sigma}+S_i))\bigg)\cup\bigg(\bigcup_{\substack{1\leq i,j\leq n\\ i\neq j}}(\sigma(a_i)+a_j+(\tau_{a_j}\sigma(S_i)+S_j))\bigg).
$$

Assume that $n=1$. Since $|\sigma(A)\stackrel{\sigma}+A|=2|A|-3$, we have
$$|\tau_{a_1}\sigma(S_1)\stackrel{\tau_{a_1}\sigma}+S_1|=2|S_1|-3.$$
By the induction hypothesis, for $s_1,s_2\in T_1$, we have $\tau_{a_1}\sigma(s_1)+s_2=\tau_{a_1}\sigma(s_2)+s_1$, i.e.,
$$
-a_1+\sigma(s_1)+a_1+s_2=-a_1+\sigma(s_2)+a_1+s_1.
$$
It follows that
$$
\sigma(a_1+s_1)+(a_1+s_2)=\sigma(a_1+s_2)+(a_1+s_1).
$$
Hence Theorem \ref{t3} is true when $n=1$.

Suppose that $n\geq 2$ and
$|S_1|\geq|S_2|\geq\cdots |S_n|$. Let $\bar{A}=\{\bar{a}_1,\ldots,\bar{a}_n\}$.
By Lemma \ref{hall}, assume that
$$
\sigma(\bar{a}_1)+\bar{a}_1,\ldots,\sigma(\bar{a}_1)+\bar{a}_n,\sigma(\bar{a}_2)+\bar{a}_{i_2},\ldots,\sigma(\bar{a}_{n})+\bar{a}_{i_n}.
$$
are distinct elements of $\sigma(\bar{A})+\bar{A}$. Then by (\ref{sigmaeh}),
\begin{align*}
|\sigma(A)\stackrel{\sigma}+A|
\geq&|\tau_{a_1}\sigma(S_1)\stackrel{\tau_{a_1}\sigma}+S_1|+\sum_{j=2}^m|\tau_{a_1}\sigma(S_1)+S_j|+\sum_{j=2}^n|\tau_{a_{i_j}}\sigma(S_j)\stackrel{\tau_{a_{i_j}}\sigma}+S_{i_j}|\notag\\
\geq&|S_1|+|S_1|-3+\sum_{j=2}^m(|S_1|+|S_j|-1)+
\sum_{j=2}^n|\tau_{a_{i_j}}\sigma(S_j)\stackrel{\tau_{a_{i_j}}\sigma}+S_{i_j}|\notag\\
\geq&|A|-2+n(|S_1|-1)+
\sum_{j=2}^n(|S_j|-1)=2|A|-3+(n-1)(|S_1|-2),
\end{align*}
where in the third inequality we use the fact
$|S\stackrel{\sigma}+T|\geq |S|-1$.
Hence we have $|\sigma(A)\stackrel{\sigma}+A|\geq2|A|-2$ if $|S_1|\geq 3$.

Thus we must have $|S_1|\leq 2$. Suppose that $|S_1|=\cdots=|S_n|=1$.
Then from $|\sigma(A)\stackrel\sigma+A|=2n-3$, we know that $|\sigma(\bar{A})\stackrel\sigma+\bar{A}|=2n-3$. By the induction hypothesis,
$\sigma(\bar{a}_i)+\bar{a}_j=\sigma(\bar{a}_j)+\bar{a}_i$ for any distinct $1\leq i,j\leq n$.
Let $X_i=a_i+S_i=\{x_i\}$. Then for distinct $1\leq i,j\leq n$, 
$|(\sigma(X_i)+X_j)\cup(\sigma(X_j)+X_i)|=1$ implies that $\sigma(x_i)+x_j=\sigma(x_j)+x_i$.

Assume that $|S_1|=2$ and
$$
\sigma(\bar{a}_1)+\bar{a}_1,\sigma(\bar{a}_1)+\bar{a}_2,\ldots,\sigma(\bar{a}_1)+\bar{a}_n,\sigma(\bar{a}_{2})+\bar{a}_{i_2},\ldots,\sigma(\bar{a}_{n})+\bar{a}_{i_n}
$$
are distinct elements of $\sigma(\bar{A})+\bar{A}$. Then
\begin{equation}\label{AsB}
|\sigma(A)\stackrel{\sigma}+A|\geq\sum_{j=1}^n(|S_1|+|S_j|-1)-2+\sum_{j=2}^n|\sigma(a_i+S_i)\stackrel{\sigma}+(a_{i_j}+S_{i_j})|\geq2|A|-3.
\end{equation}
In the first inequality of (\ref{AsB}), the equality holds only if
$$
|\tau_{a_1}\sigma(S_1)\stackrel{\tau_{a_1}\sigma}+S_1|=2|S_1|-3.
$$
And the equality holds in the second inequality of (\ref{AsB}) only if $j=i_j$ for all $2\leq j\leq n$ and
\begin{equation}\label{SsT}
|\tau_{a_j}\sigma(S_j)\stackrel{\tau_{a_j}\sigma}+S_j|=|S_j|-1.
\end{equation}

Now $\sigma(A)\stackrel{\sigma}+A$ coincides with
\begin{equation}\label{AsBs}
\bigcup_{j=2}^n(\sigma(a_1)+a_j+(\tau_{a_j}\sigma(S_1)+S_j))
\cup\bigcup_{j=1}^n(\sigma(a_j)+a_j+(\tau_{a_j}\sigma(S_j)\stackrel{\tau_{a_j}\sigma}+S_j)).
\end{equation}
Furthermore, we must have
$$
\sigma(\bar{A})\stackrel{\sigma}+\bar{A}\subseteq\{\sigma(\bar{a}_1)+\bar{a}_1,\sigma(\bar{a}_1)+\bar{a}_2,\ldots,\sigma(\bar{a}_1)+\bar{a}_n,\sigma(\bar{a}_2)+\bar{a}_2,\sigma(\bar{a}_n)+\bar{a}_n\}.
$$
Otherwise, there will exist distinct $2\leq j_1,j_2\leq n$ such that $$\sigma(a_{j_1})+a_{j_2}+(\tau_{a_{j_2}}\sigma(S_{j_1})\stackrel{\tau_{a_{j_2}}\sigma}+S_{j_2})\subseteq \sigma(A)\stackrel{\sigma}+A$$ is not included in (\ref{AsBs}).

Let
$$
J_1=\{1\leq j\leq n:\, \sigma(\bar{a}_j)+\bar{a}_j\in (\sigma(\bar{A})\stackrel{\sigma}+\bar{A})\setminus\{\sigma(\bar{a}_1)+\bar{a}_2,\cdots,\sigma(\bar{a}_1)+\sigma(\bar{a}_n)\}\},
$$
and let
$J_2=\{2,3,\ldots,n\}\setminus J_1$. We must have $|S_{j}|=2$ for all $j\in J_1$. Otherwise, if $|S_{j_0}|=1$ for some $j_0\in J_1$, then there exist
distinct $2\leq j_1,j_2\leq n$ such that $\sigma(\bar{a}_{j_0})+\bar{a}_{j_0}=\sigma(\bar{a}_{j_1})+\bar{a}_{j_2}$. By $(\ref{SsT})$, $\tau_{a_{j_0}}\sigma(S_{j_0})\stackrel{\tau_{a_{j_0}}\sigma}+S_{j_0}$ is empty.
But $$\sigma(a_{j_1}+S_{j_1})\stackrel{\sigma}+\sigma(a_{j_2}+S_{j_2})=\sigma(a_{j_1}+S_{j_1})+\sigma(a_{j_2}+S_{j_2})$$
is not empty.

We also have $n\leq 3$. Otherwise, for $n\geq 4$, it is easy to see that
$|J_2|\leq 2$ and $|J_1|\geq 2$. Hence we may
find $j_0\in J_1$ and distinct $2\leq j_1,j_2\leq m$ such that $\sigma(\bar{a}_{j_0})+\bar{a}_{j_0}=\sigma(\bar{a}_{j_1})+\bar{a}_{j_2}$ and
$\max\{|S_{j_1}|,|S_{j_2}|\}=2$. Thus in view of (\ref{SsT}), $|\tau_{a_{j_0}}\sigma(S_{j_0})\stackrel{\tau_{a_{j_0}}\sigma}+S_{j_0}|=1$ and
$$
\sigma(a_{j_1}+S_{j_1})\stackrel{\sigma}+(a_{j_2}+S_{j_2})=\sigma(a_{j_1})+a_{j_2}+(\tau_{a_{j_2}}\sigma(S_{j_2})+S_{j_2})
$$
has at least two elements. 

Now we have showed $|\sigma(A)\stackrel{\sigma}+A|\geq2|A|-3$ is impossible when $n\geq 4$. However, the case $|S_1|=2$ and $n=2,3$ are the most diffcult part in the proof of 
Theorem \ref{t2}. We shall propose its proof in the final section.

\section{Proof of Theorem \ref{t2}: The case $|S_1|=2$ and $n=2,3$}
\setcounter{equation}{0}\setcounter{Thm}{0}\setcounter{Lem}{0}\setcounter{Cor}{0}
\begin{Lem}\label{xy} Suppose that $\sigma$ is an automorphism of $G$ with odd order.\medskip\smallskip

\noindent(i) Suppose that $p(G)>2$ and $A=\{x_1,x_2\}$ and $B=\{y\}$ are two subsets of $G$. If
$$
|(\sigma(A)+B)\cup(\sigma(B)+A)|=2,
$$
then $\sigma(x_i)+y=\sigma(y)+x_i$ for $i=1,2$.\medskip\smallskip

\noindent(ii) Suppose that $p(G)>3$ and $A=\{x_1,x_2\}$ and $B=\{y_1,y_2\}$ are two subsets of $G$. 
If
$$
|\sigma(A)\stackrel{\sigma}+A|=|\sigma(B)\stackrel{\sigma}+B|=1,\quad |(\sigma(A)+B)\cup(\sigma(B)+A)|=3,
$$
then $\sigma(x_i)+y_j=\sigma(y_j)+x_i$ for $1\leq i,j\leq2$.
\end{Lem}
\begin{proof} (i) Clearly $|(\sigma(A)+B)\cup(\sigma(B)+A)|=2$ implies that $\sigma(x_1)+y$ equals either $\sigma(y)+x_1$ or $\sigma(y)+x_2$. 
Assume that
\begin{equation}\label{a1b}\sigma(x_1)+y=\sigma(y)+x_2.
\end{equation}
Then we also have \begin{equation}\label{a2b}\sigma(x_2)+y=\sigma(y)+x_1.
\end{equation}
By (\ref{a1b}), we have
$
\sigma(x_2)=\sigma(y)+x_1-y$. Substituting this to (\ref{a2b}), we get
$$
\sigma^2(x_1)+\sigma(y)=\sigma^2(y)+\sigma(x_2)=\sigma^2(y)+\sigma(y)+x_1-y,
$$
i.e., 
$$\sigma^2(x_1)=\sigma^2(y)+\sigma(y)+x_1-y-\sigma(y).$$
By an easy induction, we have
$$
\sigma^{2k}(x_1)=\sigma^{2k}(y)+\sigma^{2k-1}(y)+\cdots+\sigma(y)+x_1-y-\sigma(y)-\cdots-\sigma^{2k-1}(y).
$$
Let $h$ be the order of $\sigma$ and $k=(h|G|+1)/2$. Then
\begin{align*}
\sigma^{2k}(y)+\sigma^{2k-1}(y)+\cdots+\sigma^2(y)
=&\sigma\bigg(\sum_{j=0}^{|G|-1}(\sigma^{jh+h}(y)+\cdots+\sigma^{jh+1}(y))\bigg)\\
=&
\sigma\big(|G|(\sigma^{h}(y)+\cdots+\sigma(y))\big)=0.
\end{align*}
Similarly
$$-\sigma(y)-\cdots-\sigma^{2k-1}(y)=|G|(-\sigma(y)-\cdots-\sigma^{h}(y))=0.
$$
Hence
$$
\sigma(x_1)=\sigma^{2k}(x_1)=\sigma(y)+x_1-y,
$$
which clearly contradicts with our assumption (\ref{a1b}).

(ii) Assume that our assertion is not true. Clearly $|\sigma(A)\stackrel{\sigma}+A|=1$ implies that
\begin{equation}\label{x1x2}\sigma(x_1)+x_2=\sigma(x_2)+x_1.\end{equation}
Similarly, it follows from $|\sigma(B)\stackrel{\sigma}+B|=1$ that \begin{equation}\label{y1y2}\sigma(y_1)+y_2=\sigma(y_2)+y_1.\end{equation}
In view of (i), we may assume that $|(\sigma(A)+y_1)\cup(\sigma(y_1)+A)|=3$. That is,
$$(\sigma(A)+B)\cup(\sigma(B)+A)=\{\sigma(x_1)+y_1,\sigma(x_2)+y_1,\sigma(y_1)+x_1\}$$ or $$
(\sigma(A)+B)\cup(\sigma(B)+A)=\{\sigma(x_1)+y_1,\sigma(x_2)+y_1,\sigma(y_1)+x_2\}.$$ 
On the other hand,
$$
\sigma(A)+B=\{\sigma(x_1)+y_1,\sigma(x_2)+y_1,\sigma(x_1)+y_2,\sigma(x_2)+y_2\}.
$$
So without loss of generality, we may assume that 
\begin{equation}\label{a1b1}
\sigma(x_1)+y_1=\sigma(x_2)+y_2.
\end{equation}
Thus
$$
\sigma(A)+B=\{\sigma(x_1)+y_1,\sigma(x_2)+y_1,\sigma(x_1)+y_2\}.
$$

Now we have either $$\sigma(y_1)+x_1=\sigma(x_1)+y_2$$ or $$\sigma(y_1)+x_2=\sigma(x_1)+y_2.$$ 
Assume that 
$\sigma(y_1)+x_1=\sigma(x_1)+y_2$. 
There are the following six subcases:
$${\rm (a)}
\begin{cases}&\sigma(x_2)+y_2=\sigma(x_1)+y_1,\\
&\sigma(y_1)+x_2=\sigma(x_1)+y_1,\\
&\sigma(y_2)+x_1=\sigma(x_2)+y_1,\\
&\sigma(y_2)+x_2=\sigma(x_1)+y_2,\\
&\sigma(y_1)+x_1=\sigma(x_1)+y_2,\\
\end{cases}\qquad
{\rm (b)}\begin{cases}&\sigma(x_2)+y_2=\sigma(x_1)+y_1,\\
&\sigma(y_1)+x_2=\sigma(x_2)+y_1,\\
&\sigma(y_2)+x_1=\sigma(x_2)+y_1,\\
&\sigma(y_2)+x_2=\sigma(x_1)+y_2,\\
&\sigma(y_1)+x_1=\sigma(x_1)+y_2,\\
\end{cases}
$$

$${\rm (c)}\begin{cases}&\sigma(x_2)+y_2=\sigma(x_1)+y_1,\\
&\sigma(y_2)+x_1=\sigma(x_1)+y_1,\\
&\sigma(y_1)+x_2=\sigma(x_2)+y_1,\\
&\sigma(y_2)+x_2=\sigma(x_1)+y_2,\\
&\sigma(y_1)+x_1=\sigma(x_1)+y_2,\\
\end{cases}\qquad
{\rm (d)}\begin{cases}&\sigma(x_2)+y_2=\sigma(x_1)+y_1,\\
&\sigma(y_1)+x_2=\sigma(x_1)+y_1,\\
&\sigma(y_2)+x_1=\sigma(x_1)+y_1,\\
&\sigma(y_2)+x_2=\sigma(x_1)+y_2,\\
&\sigma(y_1)+x_1=\sigma(x_1)+y_2,\\
\end{cases}$$

$${\rm (e)}\begin{cases}&\sigma(x_2)+y_2=\sigma(x_1)+y_1,\\
&\sigma(y_1)+x_2=\sigma(x_1)+y_1,\\
&\sigma(y_2)+x_1=\sigma(x_1)+y_1,\\
&\sigma(y_2)+x_2=\sigma(x_2)+y_1,\\
&\sigma(y_1)+x_1=\sigma(x_1)+y_2,\\
\end{cases}\qquad
{\rm (f)}\begin{cases}&\sigma(x_2)+y_2=\sigma(x_1)+y_1,\\
&\sigma(y_2)+x_2=\sigma(x_1)+y_1,\\
&\sigma(y_1)+x_2=\sigma(x_2)+y_1,\\
&\sigma(y_2)+x_1=\sigma(x_2)+y_1,\\
&\sigma(y_1)+x_1=\sigma(x_1)+y_2.\\
\end{cases}$$

First, it is impossible that $\sigma(y_2)+x_1=\sigma(x_1)+y_1$ and $\sigma(y_1)+x_1=\sigma(x_1)+y_2$ hold simultaneously,
In fact, if it is true, then we have
$$
\sigma(x_1)+y_1-x_1=\sigma(y_2)=\sigma(-\sigma(x_1)+\sigma(y_1)+x_1),
$$
i.e.,
$$
\sigma^{2}(y_1)=\sigma^{2}(x_1)+\sigma(x_1)+y_1-x_1-\sigma(x_1).
$$
By the discussions in the proof of (i), we can get $\sigma(y_1)=\sigma(x_1)+y_1-x_1$.
Thus (c), (d) and (e) are omitted.

Second, $\sigma(y_2)+x_2=\sigma(x_1)+y_2$ and  $\sigma(y_1)+x_2=\sigma(x_2)+y_1$ cannot simultaneously hold.
In fact, $x_2=-\sigma(y_2)+\sigma(x_1)+y_2$ implies that
$$
\sigma(x_2)=\sigma\big(-\sigma(y_2)+\sigma(x_1)+y_2\big)=-\sigma^2(y_2)+\sigma^2(x_1)+\sigma(y_2).
$$
If $\sigma(y_1)+x_2=\sigma(x_2)+y_1$, then we have
\begin{align*}
\sigma(y_1)+(-\sigma(y_2)+\sigma(x_1)+y_2)=&(-\sigma^2(y_2)+\sigma^2(x_1)+\sigma(y_2))+y_1\\
=&-\sigma^2(y_2)+\sigma^2(x_1)+\sigma(y_1)+y_2,
\end{align*}
where in the second equality we use (\ref{y1y2}).
Thus we get
$$
y_1-y_2+x_1=-\sigma(y_2)+\sigma(x_1)+y_1.
$$
It follows that
$$
\sigma(x_1)+y_1=\sigma(y_2)+y_1-y_2+x_1=\sigma(y_1)+y_2-y_2+x_1=\sigma(y_1)+x_1.
$$
Hence (b) is impossible. Similarly, $\sigma(y_2)+x_2=\sigma(x_1)+y_1$ and  $\sigma(y_1)+x_2=\sigma(x_2)+y_1$ cannot simultaneously hold. 
This negates (f).

Finally, let us turn to (a). In view of the fifth equation of (a), we have
\begin{equation}\label{s2x1}
\sigma^2(x_1)=\sigma\big(\sigma(y_1)+x_1-y_2\big)=\sigma^2(y_1)+\sigma(x_1)-\sigma(y_2)=\sigma^2(y_1)+\sigma(y_1)+x_1-y_2-\sigma(y_2).
\end{equation}
By the fourth equation of (a), we have
$$
\sigma(x_2)=\sigma\big(-\sigma(y_2)+\sigma(x_1)+y_2\big)=-\sigma^2(y_2)+\sigma^2(x_1)+\sigma(y_2).
$$
So by (\ref{s2x1}),
$$
\sigma(x_2)+y_2=-\sigma^2(y_2)+\sigma^2(x_1)+\sigma(y_2)+y_2=-\sigma^2(y_2)+\sigma^2(y_1)+\sigma(y_1)+x_1.
$$
It follows from the first equation of (a) that
$$
\sigma(x_1)=-\sigma^2(y_2)+\sigma^2(y_1)+\sigma(y_1)+x_1-y_1.
$$
i.e.,
\begin{equation}\label{s2x1a}
\sigma^2(x_1)=-\sigma^3(y_2)+\sigma^3(y_1)+\sigma^2(y_1)+(-\sigma^2(y_2)+\sigma^2(y_1)+\sigma(y_1)+x_1-y_1)-\sigma(y_1).
\end{equation}
On the other hand, by the third equation of (a) , we have
$\sigma(x_2)=\sigma(y_2)+x_1-y_1$. And by the second equation of (a), we get
$$
\sigma\big(\sigma(y_1)+x_2\big)=\sigma^2(y_1)+(\sigma(y_2)+x_1-y_1)=\sigma^2(x_1)+\sigma(y_1),
$$
i.e.,
\begin{equation}\label{s2x1b}
\sigma^2(x_1)=\sigma^2(y_1)+\sigma(y_2)+x_1-y_1-\sigma(y_1).
\end{equation}
Combining (\ref{s2x1a}) and (\ref{s2x1b}) and recalling $\sigma^3(y_2)+\sigma^2(y_1)=\sigma^3(y_1)+\sigma^2(y_2)$ by (\ref{y1y2}), we obtain that
$$
\sigma^2(y_1)-\sigma^2(y_2)+\sigma^2(y_1)+\sigma(y_1)=\sigma^2(y_2)+\sigma(y_2),
$$
i.e.,
$$
2\sigma(-y_2+y_1)=y_2-y_1.
$$
However,
$\sigma(y_1)+y_2=\sigma(y_2)+y_1$ implies that $\sigma(-y_2+y_1)=y_1-y_2$. So we get
$$
3(y_1-y_2)=0.
$$
Hence (a) is also impossible.

Now we have proved that $\sigma(y_1)+x_1=\sigma(x_1)+y_2$ is impossible. And the case $\sigma(y_1)+x_2=\sigma(x_1)+y_2$ can be proved similarly.
\end{proof}
Let us return the proof of Theorem \ref{t3}. Suppose that $n=3$. Note that
$$
(\sigma(a_1)+a_2+(\tau_{a_2}\sigma(S_1)+S_2))\cup
(\sigma(a_2)+a_3+(\tau_{a_3}\sigma(S_2)+S_3))\cup(\sigma(a_3)+a_1+(\tau_{b_1}\sigma(S_3)+S_1))
$$
contains at least
$$
(|S_1|+|S_2|-1)+(|S_2|+|S_3|-1)+(|S_3|+|S_1|-1)=2|A|-3
$$
elements. So we have
$$
\sigma(\bar{A})\stackrel{\sigma}+\bar{A}=\{\sigma(\bar{a}_1)+\bar{a}_2,\sigma(\bar{a}_2)+\bar{a}_3,\sigma(\bar{a}_3)+\bar{a}_1\}.
$$
And by the induction hypothesis, $\sigma(\bar{a}_i)+\bar{a}_j=\sigma(\bar{a}_j)+\bar{a}_i$ for $1\leq i,j\leq 3$.
Furthermore, for $1\leq i\leq 3$, if $\sigma(\bar{a}_i)+\bar{a}_i\not\in\sigma(\bar{A})\stackrel{\sigma}+\bar{A}$, then $\tau_{a_i}\sigma(S_i)\stackrel{\tau_{a_i}\sigma}+S_i$ is empty, i.e.,
$|S_i|=1$.

Since $|S_1|=2$,  we have $\sigma(\bar{a}_1)+\bar{a}_1\in\bar{A}\stackrel{\sigma}+\bar{A}$, i.e., 
$\sigma(\bar{a}_1)+\bar{a}_1=\sigma(\bar{a}_2)+\bar{a}_3$. It follows that 
$\sigma(\bar{a}_2)+\bar{a}_2$, $\sigma(\bar{a}_3)+\bar{a}_3\not\in\bar{A}\stackrel{\sigma}+\bar{A}$
and $|S_2|=|S_3|=1$.  Let $X_i=a_i+S_i$. Assume that $X_1=\{x_1,x_2\}$, $X_2=\{y_1\}$ and $X_3=\{y_2\}$. Now we have
$$|\sigma(X_1)\stackrel\sigma+X_1|=|(\sigma(X_2)+X_3)\cup(\sigma(X_3)+X_2)|=1$$ and $$|(\sigma(X_1)+X_2)\cup(\sigma(X_2)+X_1)|=|(\sigma(X_1)+X_3)\cup(\sigma(X_3)+X_1)|=2.$$
Evidently $|\sigma(X_1)\stackrel\sigma+X_1|=1$ implies $\sigma(x_1)+x_2=\sigma(x_2)+x_1$. And 
it follows from $|(\sigma(X_2)+X_3)\cup(\sigma(X_3)+X_2)|=1$ that $\sigma(y_1)+y_2=\sigma(y_2)+y_1$.
By (i) of Lemma \ref{xy},
$|(\sigma(X_1)+X_2)\cup(\sigma(X_2)+X_1)|=2$ implies $\sigma(x_i)+y_1=\sigma(y_1)+x_i$ for $i=1,2$. Similarly,
we have $\sigma(x_i)+y_2=\sigma(y_2)+x_i$ for $i=1,2$. So Theorem \ref{t3} holds for $n=3$.

Suppose that $n=2$. Clearly we have $\sigma(\bar{a}_1)+\bar{a}_2=\sigma(\bar{a}_2)+\bar{a}_1$. 
Let $X_i=a_i+S_i$. The case $|S_2|=1$ easily follows from the discussions for $n=3$. Assume that $|S_2|=2$.
Then we have
$$
|\sigma(X_1)\stackrel{\sigma}+X_1|=|\sigma(X_2)\stackrel{\sigma}+X_2|=1
$$
and
$$
|(\sigma(X_1)\stackrel{\sigma}+X_2)\cup(\sigma(X_2)\stackrel{\sigma}+X_1)|=2.
$$
Applying (ii) of Lemma \ref{xy}, we get the desired result.
Thus the proof of Theorem \ref{t3} is concluded.\qed

\end{document}